\documentclass{article}

\usepackage{amsmath,amssymb,amsthm}
\usepackage{enumerate}
\usepackage{color}
\usepackage{authblk}
\usepackage{algorithm, algpseudocode}

\title{Multidimensional $p$-adic continued fraction algorithms}

\author[1,\thanks{saito@fun.ac.jp}]{Asaki Saito}
\author[2,\thanks{jtamura@tsuda.ac.jp}]{Jun-ichi Tamura}
\author[3,\thanks{shinichi@yasutomi-sci.toho-u.ac.jp}]{Shin-ichi Yasutomi}
\affil[1]{Future University Hakodate, Hakodate, Hokkaido 041-8655, Japan}
\affil[2]{Tsuda College, Kodaira, Tokyo 187-8577, Japan}
\affil[3]{Toho University, Funabashi, Chiba 274-8510, Japan}

\newtheorem{thm}{Theorem}[section]
\newtheorem{cor}[thm]{Corollary}
\newtheorem{lem}[thm]{Lemma}
\newtheorem{prop}[thm]{Proposition}

\theoremstyle{definition}

\theoremstyle{remark}

\begin{document}
\maketitle

\begin{abstract}
We give a new class of multidimensional $p$-adic continued fraction algorithms.
We propose an algorithm in the class for which we can expect that multidimensional $p$-adic version of
Lagrange's Theorem holds.
\end{abstract}

\footnote[0]{2010 {\it Key words and phrases.} continued fractions, multidimensional $p$-adic continued fraction algorithms}

\section{Introduction}
Throughout the paper, $p$ denotes a fixed prime number, $\Bbb{Q}_p$(resp., $\Bbb{Z}_p$)
the closure of  $\Bbb{Q}$(resp., $\Bbb{Z}$) with respect to  the $p$-adic topology.
For $x\in \Bbb{Z}_p$,  $ord_p(x)$ denotes the highest power of $p$ by which
$x$ is divided to be a $p$-adic integer.
Schneider \cite{S} has introduced the following $p$-adic continued fraction algorithm.
Let $\xi\in {\Bbb Z}_p$.
We define $\xi_1:=\xi-a_0\in p{\Bbb Z}_p$
 by choosing  $a_0\in \{0,1,\ldots,p-1\}$.
We define $\xi_n(n\geq 2)$ recursively by
\begin{align*}
\xi_{n}=\dfrac{p^{ord_p(\xi_{n-1})}}{\xi_{n-1}}-a_{n-1},
\end{align*}
where $a_{n-1}\in\{1,\ldots,p-1\}$ is chosen such that
$\xi_n\in p{\Bbb Z}_p$.
Then, we have
\begin{align*}
\xi=a_0+\cfrac{p^{ord_p(\xi_{1})}}{a_1+\cfrac{p^{ord_p(\xi_{2})}}{a_2+\cfrac{p^{ord_p(\xi_{3})}}{a_3+\ldots}}}.
\end{align*}
Weger \cite{W} has shown that some quadratic elements have not eventually periodic expansion  by Schneider's algorithm.
Ruban \cite{R} has proposed another $p$-adic continued fraction algorithm different from
Schneider's.
Ooto \cite{O} has shown a result similar to Weger's
concerning the algorithm given by Ruban.
Although Browkin \cite{Br} proposed some $p$-adic continued fraction algorithms, it has not been proved that the continued fraction expansion of every quadratic element obtained by
his algorithm is eventually periodic.
We \cite{STY} have introduced some new $p$-adic continued fraction algorithms,
and have shown $p$-adic versions of Lagrange's theorem,
i.e., if $\alpha\in \Bbb{Q}_p$ is a
quadratic element  over $\Bbb{Q}$,
the continued fractions for $\alpha$
obtained by our algorithms become periodic.
Bekki \cite{BH} has shown  a $p$-adic version of Lagrange's theorem for imaginary irrationals on his continued fraction algorithm.

It seems that there are only a few results on multidimensional $p$-adic continued fractions.
By discovering and exploiting a link between the hermitian canonical forms of certain integral matrices  and $p$-adic
numbers, Tamura \cite{T} has shown that a multidimensional $p$-adic continued fraction converges to
$(\alpha,\alpha^2,\ldots,\alpha^{n-1})$ in the $p$-adic sense without considering algorithms of continued fraction expansion, where $\alpha$ is the root of a certain  polynomial of degree $n$.
We \cite{STY2} considered  some new class of   multidimensional $p$-adic continued fractions and constructed
some explicit formulae of multidimensional continued fractions related to algebraic elements of $\Bbb{Q}_p$
over $\Bbb{Q}$.

In this paper, we propose a class $\mathcal{A}$  of multidimensional $p$-adic continued fraction algorithms such that we can expect that\\

{\bf Conjecture}.
For any  $\Bbb{Q}$-basis $\{1,\alpha_1,\ldots,\alpha_{s}\}$ of
any given field $K\subset \Bbb{Q}_p$ of degree  $[K:\Bbb{Q}]=s+1$,
the continued fraction expansion of $\overline{\alpha}=(\alpha_1,\ldots,\alpha_{s})$
by the $s$-dimensional continued fraction algorithm in the class $\mathcal{A}$ always
becomes eventually periodic(cf. \S 8). \\

In Section 4,  we shall show that  the conjecture holds for $s=1$.
We shall show that for $s>1$,
there exist infinitely many  $\overline{\alpha}\in K^s$
having eventually periodic continued fraction expansion obtained
by the algorithm in our class.
The conjecture is supported by numerical experiments (Tables 6)
for $s=2,3,4,5$ obtained
by the algorithm.
This is in contrast with experiments (Tables 1-5)
obtained
by using other algorithms.



\section{Notation and some lemmas}
We denote by $\overline{\Bbb{Q}_p}$ the algebraic closure of $\Bbb{Q}_p$, and by $\Bbb{A}_p$
 the set of algebraic elements over $\Bbb{Q}$ in $\Bbb{Q}_p$.
We put
\begin{align*}
U_p:=\Bbb{Z}_p\backslash p\Bbb{Z}_p,\
C:=\{0,1,\ldots,p-1\}.
\end{align*}

For
\begin{align*}
\alpha=\sum_{n\in \Bbb{Z}}c_np^n\in {\Bbb Q}_p\backslash\{0\}\ (c_n\in C),
\end{align*}
we define
\begin{align*}
&ord_p(\alpha):=\min\{n |c_n\ne 0\}\  (ord_p(0):=\infty),
\  |\alpha|_p:=p^{-ord_p(\alpha)}\ (|0|_p:=0),\\
 &\omega_p(\alpha):=c_0,  \ \lfloor \alpha \rfloor_p:=\Sigma_{n\in \Bbb{Z}_{\leq 0}}c_np^n
 \text{\ and\ } \langle \alpha \rangle_p:=\alpha-\lfloor \alpha \rfloor_p.
\end{align*}
For an integer $s>0$ and $\overline{\alpha}=(\alpha_1,\ldots,\alpha_s)\in \Bbb{Q}_p^s$,
we define
\begin{align*}
&ord_p(\overline{\alpha}):=\min_{1\leq i \leq s} ord_p(\alpha_i),\\
&|\overline{\alpha}|_p:=p^{-ord_p(\overline{\alpha})},\\
&\lfloor \overline{\alpha} \rfloor_p:=(\lfloor\alpha_1\rfloor_p,\ldots,\lfloor\alpha_s\rfloor_p).
\end{align*}
We define a transformation $T_{Sch}$ which is associated with Schneider's $p$-adic continued fraction on $\Bbb{Q}_p$ as follows:
for $\alpha\in \Bbb{Q}_p$ with $\alpha\ne 0$,
\begin{align*}
&T_{Sch}(\alpha):=\dfrac{p^{ord_p(\alpha)}}{\alpha}-\omega_p\left(\dfrac{p^{ord_p(\alpha)}}{\alpha}\right),\\
&T_{Sch}(0):=0.
\end{align*}
We see that  $ord_p(T_{Sch}(\alpha))>0$  for every $\alpha\in \Bbb{Q}_p$.
\\

\begin{lem}\label{l1}
Let $T$ be the transformation on $p\Bbb{Z}_p$ defined by
$T(x):=\dfrac{p^m}{x+a}$ for $x\in p\Bbb{Z}_p$, where $m$ is a positive integer and
$a\in U_p$.
Then, for $\alpha, \beta\in p\Bbb{Z}_p$,
 $|T(\alpha)-T(\beta)|_p=p^{-m}|\alpha-\beta|_p$.
\end{lem}

\begin{proof}
For $\alpha, \beta\in p\Bbb{Z}_p$,
\begin{align*}
|T(\alpha)-T(\beta)|_p=\left|\dfrac{p^m(\alpha-\beta)}{(\alpha+a)(\beta+a)}\right|_p=
p^{-m}|\alpha-\beta|_p.
\end{align*}
\end{proof}

The following Condition {\bf H} will play a key role throughout the paper.\\

We say that  $\beta \in (\Bbb{A}_p\backslash \Bbb{Q}) \cap p\Bbb{Z}_p$
satisfies  Condition {\bf H} if the  minimal polynomial over $\Bbb{Q}$
is of the form
\begin{align*}
x^n+a_1x^{n-1}+\ldots +a_n
 \in (\Bbb{Z}_p\cap {\Bbb Q})[x],\   ord_p(a_{n-1})=0 \text{\ and\ }
 ord_p(a_{n})>0.
\end{align*}

We remark that for  a polynomial
$p(x)=x^n+a_1x^{n-1}+\ldots +a_n\in (\Bbb{Z}_p\cap {\Bbb Q})[x]$
with
  $ord_p(a_{n-1})=0$ and
 $ord_p(a_{n})>0$, Hensel's Lemma says that there exists  $\alpha\in \Bbb{Q}_p$
such that $p(\alpha)=0$ and $ord_p(\alpha)=ord_p(a_{n})$.

\begin{lem}\label{l2}
Let $\beta \in (\Bbb{A}_p\backslash \Bbb{Q})\cap p\Bbb{Z}_p$.
There exists a positive integer  $m$ such that
for every integer $n\geq m$,  $T_{Sch}^n(\beta)$ satisfies Condition {\bf H}.
\end{lem}

\begin{proof}
First, we assume that every algebraic conjugate of  $\beta$ is in
$\Bbb{Q}_p$.
We denote by $\beta_1(=\beta),\ldots,\beta_n$ the algebraic conjugates of  $\beta$ and
by $\sigma_1(=identity),\ldots,\sigma_n$ the embeddings of $\Bbb{Q}(\beta)$ into $\Bbb{Q}_p$.
Then, we have for $1\leq i \leq n$
\begin{align*}
\sigma_i(T_{Sch}(\beta))=\dfrac{p^{ord_p(\beta)}}{\sigma_i(\beta)}-\omega_p\left(\dfrac{p^{ord_p(\beta)}}{\beta}\right),
\end{align*}
which implies that
for $2\leq i\leq n$,
$ord_p(\sigma_i(T_{Sch}(\beta)))=0$ if $ord_p(\sigma_i(\beta))<ord_p(\beta)$,
$ord_p(\sigma_i(T_{Sch}(\beta)))<0$ if $ord_p(\sigma_i(\beta))>ord_p(\beta)$, and
$ord_p(\sigma_i(T_{Sch}(\beta)))\geq 0$ if $ord_p(\sigma_i(\beta))=ord_p(\beta)$.
Therefore, for $2\leq i\leq n$ if for some integer $m>0$
$ord_p(\sigma_i(T^m_{Sch}(\beta)))\ne ord_p(T^m_{Sch}(\beta))$ holds, then
$ord_p(\sigma_i(T^{k}_{Sch}(\beta)))=0$ for $k\geq m+2$.
We assume that there exists some $j$ with $2\leq j\leq n$ such that
$ord_p(\sigma_j(T^m_{Sch}(\beta)))= ord_p(T^m_{Sch}(\beta))$ holds for every integer $m\geq 0$.
Then, it is not difficult to see that for every integer $m\geq 0$
\begin{align*}
\omega_p\left(\dfrac{p^{ord_p(T^m_{Sch}((\sigma_j(\beta)))}}{T^m_{Sch}(\sigma_j(\beta))}\right)
=
\omega_p\left(\dfrac{p^{ord_p(T^m_{Sch}(\beta))}}{T^m_{Sch}(\beta)}\right).
\end{align*}
 We set a transformation $T_m$$(m\in {\Bbb Z}_{\geq 0})$ on $p\Bbb{Z}_p$ by
\begin{align*}
T_m(x):=\dfrac{p^{ord_p(T^m_{Sch}(\beta))}}{x+\omega_p\left(\dfrac{p^{ord_p(T^m_{Sch}(\beta))}}{T^m_{Sch}(\beta)}\right)}
.
\end{align*}
By virtue of Lemma \ref{l1} we see that
\begin{align*}
&|\beta-\sigma_j(\beta)|_p=
|T_0\circ \cdots \circ T_{m-1}(T^m_{Sch}(\beta))-T_0\circ \cdots \circ T_{m-1}(T^m_{Sch}(\sigma_j(\beta)))|_p\\
&=p^{-\Sigma_{k=0}^{m-1} ord_p(T^k_{Sch}(\beta))}|T^m_{Sch}(\beta)-T^m_{Sch}(\sigma_j(\beta))|_p
<p^{-\Sigma_{k=0}^{m-1} ord_p(T^k_{Sch}(\beta))}.
\end{align*}
Therefore, taking $m\to \infty$ we get $|\beta-\sigma_j(\beta)|_p=0$, which contradicts
$\beta\ne \sigma_j(\beta)$.
Thus, we see that
there exists  an integer $m'>0$
such that for every integer $m''>m'$,
$ord_p(\sigma_i(T^{m''}_{Sch}(\beta)))=0$ for $i=2,\ldots,n$ and
$T^{m''}_{Sch}(\beta)$ satisfies Condition {\bf H}.
If some algebraic conjugates of  $\beta$ are not  included in
$\Bbb{Q}_p$, then considering  $\Bbb{Q}_p(\beta_1,\ldots,\beta_n)$
we have a similar proof.
\end{proof}

Lemma \ref{l2} implies  the following proposition.

\begin{prop}\label{p1}
Let $K\subset \Bbb{Q}_p$ be a finite extension of
$\Bbb{Q}$.
There exists $\alpha\in K$ which satisfies
Condition {\bf H} and $K=\Bbb{Q}(\alpha)$.
\end{prop}

\section{$c$-map}
In this section, we introduce a class of multidimensional $p$-adic continued fraction
algorithms.
Let $K\subset \Bbb{Q}_p$ be a finite extension of
$\Bbb{Q}$ of  degree $d$.
We put
\begin{align*}
&s:=d-1 (d\geq 2) \text{ and }  s:=1 (d=1),\\
&Ind:=\{1,2,\ldots,s\},\\
&D:=K^s,\ E=(p\Bbb{Z}_p)^s\cap K^s.
\end{align*}
In what follows, we always suppose that
\begin{align*}
 \overline{\alpha}=(\alpha_1,\ldots, \alpha_s)\in D=K^s(K\subset \Bbb{Q}_p)
\end{align*}
and $\overline{x}=(x_1,\ldots, x_s)\in D$.

We denote by $L(D)$ the set of linear fractional transformations on $D$.
We now introduce a map $\Phi:D \to Ind\times L(D)\times
GL(s,\Bbb{Z}_p\cap \Bbb{Q})\times (p\Bbb{Z}_p\cap \Bbb{Q})^s$.
We define
$\Phi(\overline{\alpha})$ by
$\Phi(\overline{\alpha}):=(\phi(\overline{\alpha}),F_{\overline{\alpha}},A_{\overline{\alpha}},\gamma(\overline{\alpha}))$ with $\gamma(\bar{0}):=\bar{0}$,
where
$A_{\overline{\alpha}}$ and $\gamma(\overline{\alpha})$  are arbitrarily fixed and
$F_{\overline{\alpha}}=(f_1,\ldots,f_s)$ is given as follows:\\

In the case of $\alpha_{\phi(\overline{\alpha})}\ne 0$,
we define $f_i$ for $(x_1,\ldots,x_s)\in D$ with $x_{\phi(\overline{\alpha})}\ne 0$ by
\[
f_i(x_1,\ldots,x_s) :=
\begin{cases}
\dfrac{u_{\phi(\overline{\alpha})}p^{ord_p(\alpha_{\phi(\overline{\alpha})})}}{x_{\phi(\overline{\alpha})}}-v_{\phi(\overline{\alpha})}& \text{if $i=\phi(\overline{\alpha})$},\\
\dfrac{u_i'p^k x_i}{x_{\phi(\overline{\alpha})}}-v_i'& \text{if $i\ne \phi(\overline{\alpha})$},
\end{cases}
\]
where $u_{\phi(\overline{\alpha})},v_{\phi(\overline{\alpha})},u_i'\in U_p\cap \Bbb{Q}$, $v_i' \in \Bbb{Z}_p\cap \Bbb{Q}$, and
$k=\max\{ord_p(\alpha_{\phi(\overline{\alpha})})-ord_p(\alpha_i)$,0\}.
We also assume that
$f_i(\overline{\alpha})\in p\Bbb{Z}_p$ for $1\leq i\leq s$.
Note that choices of $u_{\phi(\overline{\alpha})},v_{\phi(\overline{\alpha})},u_i',v_i'$ for $1\leq i\leq s$ are  subject to the restriction.
\
We do not define $f_i$ for $(x_1,\ldots,x_s)\in D$ with $x_{\phi(\overline{\alpha})}= 0$.\\

In the case of $\alpha_{\phi(\overline{\alpha})}= 0$,
we define $f_i$ by
\[
f_i(x_1,\ldots,x_s):=x_i.
\]

In what follows, we call $\Phi$ a $c$-map.
For a $c$-map $\Phi$ we define a transformation $T_{\Phi(\overline{\alpha})}$ on $D$ as follows:
for $x\in D$
\begin{align*}
T_{\Phi(\overline{\alpha})}(x):=A_{\overline{\alpha}}F_{\overline{\alpha}}(x)+\gamma(\overline{\alpha}).
\end{align*}


\begin{lem}\label{l2.5}
Let
$\Phi(\overline{\alpha})=(\phi(\overline{\alpha}),F_{\overline{\alpha}},A_{\overline{\alpha}},\gamma(\overline{\alpha}))$
be a $c$-map with $F=(f_1,\ldots,f_s)$ and $\gamma=(\gamma_1,\ldots,\gamma_s)$.
Let $\overline{\alpha}\in D$ with
$\alpha_{\phi(\overline{\alpha})}\ne 0$,
 then, \\$A_{\overline{\alpha}}(f_1(\overline{\alpha}),\ldots,f_s(\overline{\alpha}))+(\gamma_1(\overline{\alpha}),\ldots,\gamma_s(\overline{\alpha}))\in E$ holds.
\end{lem}
\begin{proof}
The proof is easy.
\end{proof}

\begin{lem}\label{l3}
Let $\Phi(\overline{\alpha})=(\phi(\overline{\alpha}),F_{\overline{\alpha}},A_{\overline{\alpha}},\gamma(\overline{\alpha}))$
be a $c$-map with $F_{\overline{\alpha}}=(f_1,\ldots,f_s)$.
Let
$(\beta_1,\ldots,\beta_s)^{T}=A_{\overline{\alpha}}(f_1(\overline{\alpha}),\ldots,f_s(\overline{\alpha}))+(\gamma_1(\overline{\alpha}),\ldots,\gamma_s(\overline{\alpha}))^{T}$.
If $1,\alpha_1,\ldots, \alpha_s$ are linearly independent over $\Bbb{Q}$,
then, $1,\beta_1,\ldots,\beta_s$
are linearly independent over $\Bbb{Q}$.
\end{lem}
\begin{proof}
Let $1,\alpha_1,\ldots, \alpha_s$ be linearly independent over $\Bbb{Q}$.
Then, $\frac{1}{\alpha_{\phi(\overline{\alpha})}},\frac{\alpha_1}{\alpha_{\phi(\overline{\alpha})}},\ldots, \frac{\alpha_s}{\alpha_{\phi(\overline{\alpha})}}$
are linearly independent over $\Bbb{Q}$.
Therefore, we see that $1,f_1(\overline{\alpha}),\ldots,f_s(\overline{\alpha})$
are linearly independent over $\Bbb{Q}$.
Since $A_{\overline{\alpha}}\in GL(s,\Bbb{Z}_p\cap \Bbb{Q})$,
 $1,\beta_1,\ldots,\beta_s$ are linearly independent over $\Bbb{Q}$.
\end{proof}

\begin{lem}\label{l4}
Let $\Phi(\overline{\alpha})=(\phi(\overline{\alpha}),F_{\overline{\alpha}},A_{\overline{\alpha}},\gamma(\overline{\alpha}))$
be a $c$-map with $F_{\overline{\alpha}}=(f_1,\ldots,f_s)$.
 If $\alpha_{\phi(\overline{\alpha})}\ne 0$, then
$F_{\overline{\alpha}}^{-1}:E\to D$ exists and
for $(x_1,\ldots,x_s)\in E$, $F_{\overline{\alpha}}^{-1}(x_1,\ldots,x_s)$ is given  as follows:\\
$F_{\overline{\alpha}}^{-1}(x_1,\ldots,x_s)=(g_1(x_1,\ldots,x_s),\ldots,g_s(x_1,\ldots,x_s))$, where
\begin{align*}
&\text{if $i=\phi(\overline{\alpha})$, \  there exist  $u_i,v_i\in U_p\cap \Bbb{Q}$\ such that }\\
&\hspace*{3cm}g_i(x_1,\ldots,x_s)=\dfrac{u_ip^{ord_p(\alpha_i)}}{x_i+v_i},\\
&\text{if $i\ne \phi(\overline{\alpha})$, \ there exist  $u'_i\in U_p\cap \Bbb{Q},\ w'_i\in \Bbb{Z}_p\cap \Bbb{Q}$\ such that }\\
&\hspace*{3cm}g_i(x_1,\ldots,x_s)=\dfrac{u'_ip^k (x_i+w'_i)}{x_{\phi(\overline{\alpha})}+v_{\phi(\overline{\alpha})}},\\
&\text{where $k=\min\{ord_p(\alpha_{\phi(\overline{\alpha})}),ord_p(\alpha_i)$\}}.\\
\end{align*}
\end{lem}
\begin{proof}
The proof is easy.
\end{proof}

By Lemma \ref{l4} we have

\begin{cor}\label{l4a}
Let $\Phi(\overline{\alpha})=(\phi(\overline{\alpha}),F_{\overline{\alpha}},A_{\overline{\alpha}},\gamma(\overline{\alpha}))$
be a $c$-map with $F_{\overline{\alpha}}=(f_1,\ldots,f_s)$.
 If $\alpha_{\phi(\overline{\alpha})}\ne 0$, then
$T_{\Phi(\overline{\alpha})}^{-1}:E\to D$ exists.
\end{cor}

\begin{lem}\label{l4-2}
Let $\Phi(\overline{\alpha})=(\phi(\overline{\alpha}),F_{\overline{\alpha}},A_{\overline{\alpha}},\gamma(\overline{\alpha}))$
be a $c$-map.
If $\overline{\alpha}=(\alpha_1,\ldots, \alpha_s)\in E$,
then $F_{\overline{\alpha}}^{-1}(A_{\overline{\alpha}}^{-1}(E))\subset  E$ holds.
\end{lem}

\begin{proof}
It is not difficult to see $A_{\overline{\alpha}}^{-1}(E)=E$.
We will show $F_{\overline{\alpha}}^{-1}(E)\subset  E$.
If $\alpha_{\phi(\overline{\alpha})}=0$, then $F_{\overline{\alpha}}$ is the identity map so that
$F_{\overline{\alpha}}^{-1}(E)\subset  E$.
We suppose $\alpha_{\phi(\overline{\alpha})}\ne 0$.
Let $(x_1,\ldots,x_s)\in E$ and
\begin{align*}
F_{\overline{\alpha}}^{-1}(x_1,\ldots,x_s)=(g_1(x_1,\ldots,x_s),\ldots,g_s(x_1,\ldots,x_s)).
\end{align*}
By Lemma \ref{l4} for $i=\phi(\overline{\alpha})$, we have
\begin{align*}
|g_i(x_1,\ldots,x_s)|_p=\left|\dfrac{u'_ip^{ord_p(\alpha_i)}}{x_i+v'_i}\right|_p
=p^{-ord_p(\alpha_i)},
\end{align*}
where $u'_i,v'_i\in U_p\cap \Bbb{Q}$,
and
for $i\ne \phi(\overline{\alpha})$, we have
\begin{align*}
|g_i(x_1,\ldots,x_s)|_p=\left|
\dfrac{u'_ip^k (x_i+w'_i)}{x_{\phi(\overline{\alpha})}+v'_{\phi(\overline{\alpha})}}
\right|_p\leq p^{-k},
\end{align*}
where $k=\min\{ord_p(\alpha_{\phi(\overline{\alpha})}), ord_p(\alpha_i)\}$,  $u'_i,v'_i\in U_p\cap \Bbb{Q}$ and $\ w'_i\in \Bbb{Z}_p\cap \Bbb{Q}$.
\end{proof}

\begin{lem}\label{l5}
Let $\Phi(\overline{\alpha})=(\phi(\overline{\alpha}),F_{\overline{\alpha}},A_{\overline{\alpha}},\gamma(\overline{\alpha}))$
be a $c$-map.
If $\alpha_{\phi(\overline{\alpha})}\ne0$ for $\overline{\alpha}=(\alpha_1,\ldots, \alpha_s)\in E$,
then for $\overline{x}=(x_1,\ldots,x_s), \overline{y}=(y_1,\ldots,y_s)\in E$ we have
$|T_{\Phi(\overline{\alpha})}^{-1}(\overline{x})-
T_{\Phi(\overline{\alpha})}^{-1}(\overline{y})|_p\leq p^{-j}|\overline{x}-\overline{y}|_p$,
 where $j=\min\{ord_p(\alpha_i)|1\leq i \leq s\}$.
If $\alpha_{\phi(\overline{\alpha})}=0$,
for $\overline{x}=(x_1,\ldots,x_s), \overline{y}=(y_1,\ldots,y_s)\in E$
we have $|T_{\Phi(\overline{\alpha})}^{-1}(\overline{x})-
T_{\Phi(\overline{\alpha})}^{-1}(\overline{y})|_p=|\overline{x}-\overline{y}|_p$.
\end{lem}

\begin{proof}
Let $\overline{\alpha}=(\alpha_1,\ldots, \alpha_s)\in E$.
We assume that $\overline{x}=(x_1,\ldots,x_s), \overline{y}=(y_1,\ldots,y_s)\in E$.
First, we suppose $\alpha_{\phi(\overline{\alpha})}\ne 0$.
Let $F_{\overline{\alpha}}^{-1}=(g_1,\ldots,g_s)$.
By Lemma \ref{l4} we see that
for $i=\phi(\overline{\alpha})$,
\begin{align*}
&|g_i(\overline{x})-g_i(\overline{y})|_p=\left|\dfrac{u'_ip^{ord_p(\alpha_i)}}{x_i+v'_i}-
\dfrac{u'_ip^{ord_p(\alpha_i)}}{y_i+v'_i}\right|_p\\
&=\dfrac{|u'_ip^{ord_p(\alpha_i)}|_p|x_i-y_i|_p}{|x_i+v'_i|_p|y_i+v'_i|_p}
= p^{-ord_p(\alpha_i)}|x_i-y_i|_p,
\end{align*}
since $u'_i,v'_i\in U_p\cap \Bbb{Q}$.
For $i\ne \phi(\overline{\alpha})$ and $1\leq i\leq s$,
\begin{align*}
&|g_i(\overline{x})-g_i(\overline{y})|_p=
\left|\dfrac{u'_ip^k (x_i+w'_i)}{x_{\phi(\overline{\alpha})}+v'_{\phi(\overline{\alpha})}}-\dfrac{u'_ip^k (y_i+w'_i)}{y_{\phi(\overline{\alpha})}+v'_{\phi(\overline{\alpha})}}  \right|_p\\
&=\dfrac{|u'_ip^{k}|_p|
(x_i-y_i)(v'_{\phi(\overline{\alpha})}+y_{\phi(\overline{\alpha})})
-(x_{\phi(\overline{\alpha})}-y_{\phi(\overline{\alpha})})(w'_i+y_i)|_p}{|x_{\phi(\overline{\alpha})}+v'_{\phi(\overline{\alpha})}|_p
|y_{\phi(\overline{\alpha})}+v'_{\phi(\overline{\alpha})}|_p}\\
&\leq p^{-k}|\overline{x}-\overline{y}|_p,
\end{align*}
where $k=\min\{ord_p(\alpha_{\phi(\overline{\alpha})}),ord_p(\alpha_i)\}$,  $u'_i,v'_i\in U_p\cap \Bbb{Q}$ and $\ w'_i\in \Bbb{Z}_p\cap \Bbb{Q}$.
Since $A_{\overline{\alpha}}^{-1}(\overline{x}-\gamma(\overline{\alpha})),
A_{\overline{\alpha}}^{-1}(\overline{y}-\gamma(\overline{\alpha}))\in E$ and
$|A_{\overline{\alpha}}^{-1}(\overline{x}-\gamma(\overline{\alpha}))-A_{\overline{\alpha}}^{-1}(\overline{y}-\gamma(\overline{\alpha}))|_p
=|\overline{x}-\overline{y}|_p$, we have
\begin{align*}
&|T_{\Phi(\overline{\alpha})}^{-1}(\overline{x}))-
T_{\Phi(\overline{\alpha})}^{-1}(\overline{y}))|_p\\
&=
|F_{\overline{\alpha}}^{-1}(A_{\overline{\alpha}}^{-1}(\overline{x}-\gamma(\overline{\alpha})))-
F_{\overline{\alpha}}^{-1}(A_{\overline{\alpha}}^{-1}(\overline{y}-\gamma(\overline{\alpha})))|_p\leq p^{-j}|\overline{x}-\overline{y}|_p,
\end{align*}
where  $j=\min\{ord_p(\alpha_i)|1\leq i \leq s\}$.
Secondly, we suppose $\alpha_{\phi(\overline{\alpha})}=0$.
Then, $F_{\overline{\alpha}}$ is the identity map.
Therefore, we have
$|T_{\Phi(\overline{\alpha})}^{-1}(\overline{x}))-
T_{\Phi(\overline{\alpha})}^{-1}(\overline{y}))|_p=|\overline{x}-\overline{y}|_p$.
\end{proof}
To each $c$-map we associate a $p$-adic continued fraction map.
Let $\Phi(\overline{\alpha})=(\phi(\overline{\alpha}),F_{\overline{\alpha}},A_{\overline{\alpha}},\gamma(\overline{\alpha}))$
be a $c$-map.
We set $\overline{\alpha}^{(0)}:=\overline{\alpha}$, and define
$\overline{\alpha}^{(1)},\overline{\alpha}^{(2)},\ldots$  inductively as follows:
We suppose that $\overline{\alpha}^{(n)}$ for $n\in {\Bbb Z}_{\geq 0}$ is defined.
We set $\overline{\alpha}^{(n+1)}:=T_{\Phi(\overline{\alpha}^{(n)})}(\overline{\alpha}^{(n)})$.
We say that $\overline{\alpha}$  has a $\Phi$ continued fraction expansion
$\{\Phi(\overline{\alpha}^{(0)}),\Phi(\overline{\alpha}^{(1)}),\ldots \}$.
We refer to $\overline{\alpha}^{(n)}=(\alpha^{(n)}_1,\ldots,\alpha^{(n)}_s)$ as the $n$-th remainder of $\overline{\alpha}$.
We define the $n$-th convergent $\pi(\overline{\alpha};n)$ by
\begin{align*}
\pi(\overline{\alpha};n):=
T_{\Phi(\overline{\alpha}^{(0)})}^{-1}\cdots T_{\Phi(\overline{\alpha}^{(n-1)})}^{-1}(\overline{0}),\ \ n>0.
\end{align*}
We remark that
$\pi(\overline{\alpha};n)\in {\Bbb Q}^s$ for every $n\geq 0$.

We say that $\overline{\alpha}$  has a periodic $\Phi$ continued fraction expansion
if $\overline{\alpha}^{(m_1)}=\overline{\alpha}^{(m_2)}$ holds for some  $m_1,m_2\in {\Bbb Z}_{\geq 0}$ with
$m_1\ne m_2$.
We say that $\overline{\alpha}$  has a finite $\Phi$ continued fraction expansion
if $\overline{\alpha}^{(m)}=0$ holds for some  $m\in {\Bbb Z}_{\geq 0}$.
We say that $\overline{\alpha}$  has an infinite $\Phi$ continued fraction expansion
if $\overline{\alpha}$  does not have a finite $\Phi$ continued fraction expansion.

\begin{thm}\label{t1}
Let $\Phi(\overline{\alpha})=(\phi(\overline{\alpha}),F_{\overline{\alpha}},A_{\overline{\alpha}},\gamma(\overline{\alpha}))$
be a $c$-map.
If $\alpha^{(n)}_{\phi(\overline{\alpha}^{(n)})}$ are not equal to $0$ for infinitely many $n$,
then $\displaystyle \lim_{n\to \infty} \pi(\overline{\alpha};n)=\overline{\alpha}$.
\end{thm}
\begin{proof}
We suppose  that
 $\phi(\overline{\alpha}^{(n)})$ are not equal to $0$ for infinitely many $n$.
By Lemma \ref{l2.5} there exists an integer $m\geq 0$
such that
$\overline{\alpha}^{(m)}\in E$ holds.
By Lemma \ref{l5} we have
\begin{align*}
&|\overline{\alpha}^{(m)}-\pi(\overline{\alpha}^{(m)};n)|_p\\
&=\left|T_{\Phi(\overline{\alpha}^{(m)})}^{-1}\cdots T_{\Phi(\overline{\alpha}^{(m+n-1)})}^{-1}(\overline{\alpha}^{(m+n)})
-\right.\\
&\left.T_{\Phi(\overline{\alpha}^{(m)})}^{-1}\cdots T_{\Phi(\overline{\alpha}^{(m+n-1)})}^{-1}(\overline{0})
\right|_p\\
&\leq p^{-(j_0+\ldots+j_{n-1})}|\overline{\alpha}^{(m+n)}|_p
<p^{-(j_0+\ldots+j_{n-1})},
\end{align*}
where for $i\in {\Bbb Z}_{\geq 0}$ $j_i:=\min\{ord_p(\alpha^{(m+i)}_k)|1\leq k\leq s\}$
 if $\alpha^{(m+i)}_{\phi(\overline{\alpha}^{(m+i)})}\ne 0$ and
$j_i:=0$
 if $\alpha^{(m+i)}_{\phi(\overline{\alpha}^{(m+i)})}= 0$.
Since $\sum_{k=0}^\infty j_k=\infty$, we have
$\displaystyle \lim_{n\to \infty} \pi(\overline{\alpha}^{(m)};n)=\overline{\alpha}^{(m)}$.
Therefore, we have
\begin{align*}
&\lim_{n\to \infty}\pi(\overline{\alpha}^{(0)};m+n)\\
&=\lim_{n\to \infty}
T_{\Phi(\overline{\alpha}^{(0)})}^{-1}\cdots
T_{\Phi(\overline{\alpha}^{(m-1)})}^{-1}
(\pi(\overline{\alpha}^{(m)};n))\\
&=T_{\Phi(\overline{\alpha}^{(0)})}^{-1}\cdots
T_{\Phi(\overline{\alpha}^{(m-1)})}^{-1}
(\overline{\alpha}^{(m)})=\overline{\alpha}.
\end{align*}
\end{proof}

\begin{prop}\label{p3}
Let $\Phi(\overline{\alpha})=(\phi(\overline{\alpha}),F_{\overline{\alpha}},A_{\overline{\alpha}},\gamma(\overline{\alpha}))$ $(\overline{\alpha}
=(\alpha_1,\ldots, \alpha_s)\in D)$
be a $c$-map.
If $\overline{\alpha}$ has a finite $\Phi$ continued fraction, then  $\displaystyle \lim_{n\to \infty} \pi(\overline{\alpha};n)=\overline{\alpha}$.
\end{prop}

\begin{proof}
Let $\overline{\alpha}\in D$ and $\overline{\alpha}$ have a finite $\Phi$ continued fraction.
Then, there exists an integer $m\geq 0$ such that
$\overline{\alpha}^{(m)}=0$.
It is clear that
$\overline{\alpha}^{(n)}=0$ for every $n\geq m$.
Therefore, for every $n\geq m$ we see that $\overline{\alpha}^{(0)}=
T_{\Phi(\overline{\alpha}^{(0)})}^{-1}\cdots
T_{\Phi(\overline{\alpha}^{(n-1)})}^{-1}(\overline{0})
=\pi(\overline{\alpha};n)$.
Thus, we obtain the proposition.
\end{proof}

\section{Specific algorithms}
We have introduced a class of multidimensional $p$-adic continued fraction algorithms in Section 3. In this section, we define some particular algorithms in the class.
Let $K$ be the same as in Section 3.

By Proposition \ref{p1} there exists an element $z$ in $K$ which satisfies Condition {\bf H} and $K={\Bbb Q}(z)$.
Hereafter, we suppose that $z\in K$ satisfies Condition {\bf H}, $K={\Bbb Q}(z)$
and $\epsilon\in \{-1,1\}$.
We begin  by defining a particular $c$-map.


\subsection{c-map  $\Phi^{[\epsilon]}_0$}
For $j\in Ind$,
 we define linear fractional transformations
 $G_j^{[\overline{\alpha},\epsilon]}=(g_1^{[\overline{\alpha},\epsilon];(j)},\ldots,g_s^{[\overline{\alpha},\epsilon];(j)})$ on $D$
 as follows:\\
if $\alpha_j\ne 0$,
 for $\overline{x}:=(x_1,\ldots, x_s)\in D$ and $i\in Ind$,
\[
g_i^{[\overline{\alpha},\epsilon];(j)}(\overline{x}) :=
\begin{cases}
\dfrac{\epsilon p^{ord_p(\alpha_j)}}{x_j}-\omega_p\left(\dfrac{\epsilon p^{ord_p(\alpha_j)}}{\alpha_j}\right)& \text{if $i=j$},\\
\dfrac{\epsilon p^k x_i}{x_{j}}-\omega_p\left( \dfrac{\epsilon p^k \alpha_i}{\alpha_{j}} \right)& \text{if $i\ne j$},
\end{cases}
\]
where $k=\max\{ord_p(\alpha_{j})-ord_p(\alpha_i),0\}$.\\
If $\alpha_j= 0$, then
\[
G_j^{[\overline{\alpha},\epsilon]}(\overline{x}):=\overline{x}.
\]
We denote by $S=(s_{ij})\in GL(s,\Bbb{Z}_p\cap \Bbb{Q})$ the matrix defined by
\begin{align*}
  s_{ij}:=\begin{cases}\delta_{(i+1)j}\hspace*{1cm}&\text{for $1\leq i \leq s-1$, $1\leq j \leq s$},\\
  \delta_{1j}\hspace*{1cm}&\text{for $ i=s$, $1\leq j \leq s$},
     \end{cases}
\end{align*}
where $\delta_{ii}:=1$ and $\delta_{ij}:=0$ for $i\ne j$($i,j\in Ind)$.
We give a definition of  a $c$-map.

{\bf Definition of $\Phi^{[\epsilon]}_0$}
\begin{align}\label{phi0}
\Phi^{[\epsilon]}_0(\overline{\alpha}):=(1,G^{[\overline{\alpha},\epsilon]}_{1},S,\bar{0}),
 \text{ for \ }\overline{\alpha}\in D.
\end{align}
We remark that for $s=1$ the $\Phi^{[1]}_0$ continued fraction algorithm coincides with Schneider's
continued fraction algorithm.

\subsection{c-map  $\Phi^{[\epsilon,z]}_1$}
We give the definition of a $c$-map $\Phi_1$.
We assume that $K\ne\Bbb{Q}$.
We define linear fractional transformations $H^{[\overline{\alpha},\epsilon,z]}_j=(h_1^{[\overline{\alpha},\epsilon,z];(j)},\ldots,h_s^{[\overline{\alpha},\epsilon,z];(j)})$ on $D$
with $j\in Ind$ as follows:\\
The element
$g_i^{[\overline{\alpha},\epsilon];(j)}(\overline{\alpha})\in K$ is  uniquely written  $g_i^{[\overline{\alpha},\epsilon];(j)}(\overline{\alpha})=a_0+a_1z+ \ldots+a_sz^s$ where
$a_i\in {\Bbb Q}$ for $0\leq i\leq s$.
We define $h_i^{[\overline{\alpha},\epsilon,z];(j)}(\overline{x})$($i,j\in Ind)$ as
\begin{align*}
h_i^{[\overline{\alpha},\epsilon,z];(j)}(\overline{x}):=g_i^{[\overline{\alpha},\epsilon];(j)}(\overline{x})-\left\langle a_0\right\rangle_p.
\end{align*}

{\bf Definition of $\Phi^{[\epsilon,z]}_1$}
\begin{align}
\Phi^{[\epsilon,z]}_1(\overline{\alpha}):=(1,H{}_{1}^{[\overline{\alpha},\epsilon,z]},S,\bar{0}),
\text{for\ } \overline{\alpha}\in D.
\end{align}

{\bf Example.} Let $z\in 3{\Bbb Z}_3$ be the root of $x^2+x+3=0$.
We take $\alpha=2z+3$.
Then,
\begin{align*}
&h_1^{[\alpha,1,z];(1)}(\alpha)=\dfrac{3}{\alpha}-\dfrac{1}{5}=-\dfrac{2}{5}z,\\
&h_1^{[-\frac{2}{5}z,1,z];(1)}\left(-\dfrac{2}{5}z\right)=\dfrac{3}{-\frac{2}{5}z}-\dfrac{5}{2}=\dfrac{5}{2}z,\\
&h_1^{[\frac{5}{2}z,1,z];(1)}\left(\frac{5}{2}z\right)=\dfrac{3}{\frac{5}{2}z}+\dfrac{2}{5}=-\dfrac{2}{5}z.
\end{align*}
Therefore, $\alpha$ has the  periodic 	$\Phi^{[1,z]}_1$ continued fraction expansion given by
\begin{align*}
\alpha=\cfrac{3}{\dfrac{1}{5}+\cfrac{3}{\dfrac{5}{2}+\cfrac{3}{-\dfrac{2}{5}+\cfrac{3}{\dfrac{5}{2}+\ldots}}}}.
\end{align*}

\subsection{c-map  $\Phi^{[\epsilon,z],(n)}_2$}
We need some definitions to introduce the $c$-map $\Phi^{[\epsilon,z],(n)}_2$ for $n\in {\Bbb Z}_{\geq 1}$.
For $\alpha\in K$, $\alpha$ is  uniquely  written as $\alpha=a_0+a_1z+ \ldots+a_sz^s$ with
$a_i\in {\Bbb Q}$ for $0\leq i \leq s$.
We define $denom_z(\alpha)$ and $denom_z(\overline{\alpha})$ by
\begin{align*}
&denom_z(\alpha):=\min \{|d|\ |d\in {\Bbb Z},d(a_0+a_1x+ \ldots+a_sx^s)\in {\Bbb Z}[x] \},\\
&denom_z(\overline{\alpha}):=\max \{denom_z(\alpha_i)|1\leq i\leq s  \}.
\end{align*}
We define $v^{(1)}_{[\epsilon,z]}:D\to {\Bbb Z}$
\begin{align*}
v^{(1)}_{[\epsilon,z]}(\overline{\alpha}):=\min \{denom_z(H^{[\overline{\alpha},\epsilon,z]}_i(\overline{\alpha}))|1\leq i\leq s  \}.
\end{align*}
We define $v^{(n)}_{[\epsilon,z]}:D\to {\Bbb Z}$ $(n=2,3,\ldots)$ recursively as
\begin{align*}
v^{(n)}_{[\epsilon,z]}(\overline{\alpha}):
=\min \{denom_z(H_i^{[\overline{\alpha},\epsilon,z]}(\overline{\alpha}))v^{(n-1)}_{[\epsilon,z]}(H_i^{[\overline{\alpha},\epsilon,z]}(\overline{\alpha}))|1\leq i\leq s  \}.
\end{align*}
Define $\phi^{(n)}_{[\epsilon,z]}:D\to Ind$ for $n\in {\Bbb Z}_{\geq 1}$  by
\begin{align*}
\phi^{(n)}_{[\epsilon,z]}(\overline{\alpha}):=
\min \{i\in Ind| v^{(n+1)}_{[\epsilon,z]}(\overline{\alpha})=denom_z(H^{[\overline{\alpha},\epsilon,z]}_i(\overline{\alpha}))v^{(n)}_{[\epsilon,z]}(H^{[\overline{\alpha},\epsilon,z]}_i(\overline{\alpha})) \}.
\end{align*}
We give

{\bf Definition of $\Phi^{[\epsilon,z],(n)}_2$}
\begin{align}
\Phi^{[\epsilon,z],(n)}_2(\overline{\alpha}):=(\phi^{(n)}_{[\epsilon,z]}(\overline{\alpha}),H^{[\overline{\alpha},\epsilon,z]}_{\phi^{(n)}_{[\epsilon,z]}(\overline{\alpha})},id,\bar{0}) ,
\end{align}
where $n\in {\Bbb Z}_{\geq 1}$ and $id$ is the identity matrix.

\subsection{c-map  $\Phi^{[z]}_3$}
Finally, we shall define  a $c$-map $\Phi^{[\epsilon,z]}_3$ in a few pages.
For
\begin{align*}
\alpha=\sum_{n\in \Bbb{Z}}c_np^n\in {\Bbb Q}_p\backslash\{0\}\ (c_n\in C),
\end{align*}
We define $\lfloor \alpha:m\rfloor_p$ and $\langle \alpha:m\rangle_p$
as
\begin{align*}
&\lfloor \alpha:m\rfloor_p:=\Sigma_{n\leq m, n\in \Bbb{Z}}c_np^n,\\
&\langle \alpha:m\rangle_p:=\Sigma_{n>m, n\in \Bbb{Z}}c_np^n.
\end{align*}
We denote by $M(n;\Bbb{Q})$  the set of  $n\times n$  matrices
with entries in $\Bbb{Q}$.
We say that
$M=(m_{ij})\in M(n;\Bbb{Q})$ is {\it $p$-reduced} if $M$ satisfies that
for  every integer $i$ with $1\leq i\leq n$ there exist  a unique integer $u(i)$ with $0\leq u(i)\leq n$
 such that   \\
(1)  $m_{ik}=0$ for every integer $k$ with $1\leq k\leq u(i)$,  \\
(2) if $u(i)\ne n$, then $m_{i\ \hspace*{-0.1cm}u(i)+1}\in \{p^l|l\in {\Bbb Z}\}$,
and   $m_{k\ \hspace*{-0.1cm}u(i)+1}=0$ for every integer $k$ with $i<k\leq n$, \\
(3) if $i>1$, $u(i)\geq u(i-1)$,\\
and\\
(4) if $u(i)\ne n$, then for every integer $j$ with $1\leq j<i$
\begin{align*}
\langle m_{j\ \hspace*{-0.1cm}u(i)+1}:ord_p(m_{i\ \hspace*{-0.1cm}u(i)+1})-1\rangle_p=0.
\end{align*}

A matrix in $M(n;\Bbb{Q})$ is converted to a $p$-reduced matrix by using  the following row operations:\\
(a) Switch two rows,\\
(b) Multiply a row by an element of $U_p\cap \Bbb{Q}$,\\
(c) Add a multiple of a row  by an element of $\Bbb{Z}_p\cap \Bbb{Q}$ to another row.

We give an algorithm by which we can find a  $p$-reduced matrix
for any given  $M=(c_{ij})\in M(n;\Bbb{Q})$.
The following program describes such an algorithm.

\newpage

\begin{algorithm}
\caption{$p$-reduction algorithm}
\begin{algorithmic}[1]
   \Require \Statex $M:=(c_{ij})\in M(n;\Bbb{Q})$;
   \Ensure  \Statex $p$-reduced matrix  $M':=(c_{ij})\in M(n;\Bbb{Q})$;
   \For{$i=1,\ldots,n$}
         \State ${\bf b}_i:=(c_{i1},\ldots,c_{in})$;
   \EndFor
   \State $k_1:=1$;
   \State $k_2:=1$;
   \Repeat
         \If{there exists $i$ with $c_{ik_2}\ne 0$ for $k_1\leq i\leq n$}
            \State let $m (k_1\leq m \leq n)$ be the least integer which satisfies
            \State $|c_{mk_2}|_p=\max\{ |c_{ik_2}|_p|k_1\leq i\leq n$, $c_{ik_2}\ne 0\}$;
            \If{$m\ne k_1$}
               \State swap  ${\bf b}_{k_1}$ and  ${\bf b}_m$;
            \EndIf
            \For{$i=k_1+1,\ldots,n$}
　　　　　　　　\State ${\bf b}_i={\bf b}_i-\dfrac{c_{ik_2}}{c_{k_1k_2}}{\bf b}_{k_1}$ ;　　　　　　
            \EndFor
            \State ${\bf b}_{k_1}=\dfrac{1}{|c_{k_1k_2}|_pc_{k_1k_2}}{\bf b}_{k_1}$;
            \For{$i=1,\ldots,k_1-1$}
\State　${\bf b}_i={\bf b}_i-\dfrac{\langle c_{ik_2}:ord_p(c_{k_1k_2})-1\rangle_p}{c_{k_1k_2}}{\bf b}_{k_1}$ ;            \EndFor
   　　　　\State $k_1=k_1+1$;
　　　　　 \State $k_2=k_2+1$;　
         \EndIf
   \Until{$k_2\leq n$;}
\end{algorithmic}
\end{algorithm}

When $M\in M(n;{\Bbb Q})$ is converted by the {\it $p$-reduction} algorithm to $M'\in M(n;{\Bbb Q})$,
there exists $N\in GL(n,{\Bbb Z}_p\cap {\Bbb Q})$ associated with the algorithm
such that $M'=NM$ and we denote $N$ by $pr(M)$.
One can prove the following lemma in the usual way.

\begin{lem}\label{la1}
For $M\in M(n;{\Bbb Q})$ and $N_1, N_2\in GL(n,{\Bbb Z}_p\cap {\Bbb Q})$
$N_1M$ and $N_2M$ are $p$-reduced, then $N_1M=N_2M$ holds.
\end{lem}

We give an example.\\

\noindent
Example.
Let $p=2$.
$\displaystyle M=\begin{pmatrix}10&3/2\\-5&7\end{pmatrix}$
is converted by the $p$-reduction algorithm to
$\displaystyle M'=\begin{pmatrix}1&0\\0&1/2\end{pmatrix}$ and
$pr(M):=\displaystyle \begin{pmatrix}14/155&-3/155\\1/31&2/31\end{pmatrix}$.
\\

We recall that  $z\in K$ satisfies  Condition {\bf H} and $K={\Bbb Q}(z)$.
For $\overline{\alpha}=(\alpha_1,\ldots, \alpha_s)\in D$
the $M_{\overline{\alpha}}=(m_{ij})\in  M(s\times (s+1);\Bbb{Q})$ is defined by
$\overline{\alpha}=M_{\overline{\alpha}}(z^s,\ldots,z,1)$ and
$M'_{\overline{\alpha}}\in M_{s}(\Bbb{Q})$ is given by
$M'_{\overline{\alpha}}:=(m_{ij})_{1\leq i\leq s,1\leq j\leq s}$.
We define a map $\tau_z:D\to M_{s}({\Bbb Z}_p\cap \Bbb{Q})$ as
\begin{align*}
\tau_z(\overline{\alpha}):=pr(M'_{\overline{\alpha}}).
\end{align*}
We define $\gamma' (\overline{\alpha})\in (p\Bbb{Z}_p\cap Q)^s$ by
\begin{align*}
\gamma' (\overline{\alpha})
:=(-\langle l_{1s+1}\rangle_p,\ldots,-\langle l_{ss+1}\rangle_p),
\end{align*}
where $(l_{ij})_{1\leq i\leq s,1\leq j\leq s+1}=pr(M'_{\overline{\alpha}})M_{\overline{\alpha}}$.
We define a $c$-map $\Phi^{[z]}_3$ by

{\bf Definition of $\Phi^{[z]}_3$}
\begin{align}
\Phi^{[z]}_3(\overline{\alpha}):=(s,G^{[\overline{\alpha},\epsilon,z]}_{s},\tau_z(G^{[\overline{\alpha},1,z]}_{s}(\overline{\alpha})),\gamma' (\overline{\alpha})),  \text{ for\ }
\overline{\alpha}\in D.
\end{align}

\section{Rational and Quadratic Cases}
In this section, we consider the periodicity of the expansion obtained by
our algorithms in  specific cases.
It is well-known that
every rational number has a finite regular continued fraction expansion.
We will see that a similar result holds for the $\Phi_0^{[-1]}$ continued fraction  algorithm.

\begin{prop}\label{p2}
Let $K$ be  $\Bbb{Q}$.
Then,
every rational number $\alpha$ has a finite  $\Phi_0^{[-1]}$ continued fraction expansion.
\end{prop}
\begin{proof}
We define the height of rational number by the summation of the absolute value of its numerator and  the absolute that of its
denominator, which is denoted by $height()$.
Let $\alpha$ be a rational number.
By Lemma \ref{l2.5} we assume $\alpha\in p{\Bbb Z}_p$.
If $\alpha=0$, then $\alpha$ has obviously a finite  $\Phi_0^{[-1]}$ continued fraction expansion.
Let $\alpha=\frac{m_1p^k}{m_2}\ne 0$ where $m_1,m_2$ are relatively prime integers with $ord_p(m_1)=ord_p(m_2)=0$ and
 $k\in {\Bbb Z}_{\geq 1}$.
Then, we see that the next remainder  $\alpha_1$ is $\frac{-m_2-jm_1}{m_1}$ where $j=\omega_p(-\frac{m_2}{m_1})$.
We have
\begin{align*}
&height(\alpha)=|m_1p^k|+|m_2|\geq p|m_1|+|m_2|
\geq |-m_2-jm_1|+|m_1|\\
&=height(\alpha_1),
\end{align*}
in which the equality holds for $\alpha>0$, $j=p-1$ and $k=1$.
If the sign of $\alpha$ is positive, the sign of the next remainder is minus.
Therefore,
if $\alpha$ has an infinite  $\Phi_0^{[-1]}$ continued fraction expansion, we see
 $height(\alpha)>height(\alpha_2)>height(\alpha_4)>\ldots $, which is a contradiction.
\end{proof}

For the $\Phi_0^{[1]}$ continued fraction  algorithm on $\Bbb{Q}$, which coincides
with Schneider's continued fraction  algorithm, Bundschuh \cite{B} showed that
every rational number has an infinite periodic expansion or a finite expansion.

Next, we consider  quadratic cases. Weger \cite{W} showed that
some quadratic elements have  non periodic Schneider's continued fraction expansions.
Some numerical experiments given in Section 8 say that
some quadratic elements possibly have non periodic  expansion
by the $\Phi_0^{[-1]}$ continued fraction algorithm.

For the $\Phi_1^{[\epsilon,z]}$ continued fraction, we give following Lemma.

\begin{lem}\label{l7}
Let $K$ be a quadratic field over $\Bbb{Q}$.
Let $z\in K$ satisfy Condition {\bf H} and $K={\Bbb Q}(z)$.
For $q\in {\Bbb Q}$ with $ord_p(qz)>0$, $qz$ has a periodic $\Phi_1^{[\epsilon,z]}$ continued fraction expansion.
\end{lem}
\begin{proof}
Let $x^2+ux+vp^k$ be a minimal polynomial of $z$, where $u,v\in {\Bbb Q}$, $ord_p(u)=ord_p(v)=0$ and $k>0$.
Let $v=\dfrac{v_1}{v_2}$ and $q=\dfrac{m_1}{m_2}$, where $v_1, v_2$ are relatively prime integers with $v_2>0$
and $m_1, m_2$ are relatively prime integers with $m_2>0$.
First, we suppose that
$ord_p(q)=0$.
Since $ord_p(z)=k$, we have
\begin{align*}
G^{[qz,\epsilon]}(qz)=\epsilon\dfrac{m_2p^k}{m_1z}-j=-\epsilon \dfrac{m_2v_2z}{m_1v_1}-\epsilon\dfrac{m_2uv_2}{m_1v_1}-j,
\end{align*}
where $j=\omega_p(\epsilon\frac{m_2p^k}{m_1z})$.
Then, we have
\begin{align*}
H^{[qz,\epsilon,z]}(qz)=-\epsilon \dfrac{m_2v_2z}{m_1v_1}+\left\lfloor-\epsilon\dfrac{m_2uv_2}{m_1v_1}-j\right\rfloor_p.
\end{align*}
Since $H^{[qz,\epsilon,z]}(qz), \ -\epsilon \dfrac{m_2zv_2}{m_1v_1}\in p{\Bbb Z}_p$, so that
$\left\lfloor-\epsilon\dfrac{m_2uv_2}{m_1v_1}-j\right\rfloor_p=0$.
Then, we have
\begin{align*}
H^{[qz,\epsilon,z]}(qz)=-\epsilon \dfrac{m_2v_2z}{m_1v_1}.
\end{align*}
Therefore, the next remainder is
\begin{align*}
-\epsilon \dfrac{-\epsilon m_1v_1z}{m_2v_2}\dfrac{v_2}{v_1}=qz.
\end{align*}
Hence, $qz$ has a periodic $\Phi_1^{[\epsilon,z]}$ continued fraction expansion.
Next, we suppose that
$ord_p(q)=g>0$.
Let $m_1=m_1'p^g$, where $m_1'\in {\Bbb Z}$.
Then, we have
\begin{align*}
G^{[qz,\epsilon]}(qz)=\epsilon\dfrac{m_2p^{k+g}}{m_1z}-j=-\epsilon \dfrac{m_2v_2z}{m_1'v_1}-\epsilon\dfrac{m_2uv_2}{m_1'v_1}-j,
\end{align*}
where $j=\omega(\epsilon\frac{m_2p^{k+g}}{m_1z})$.
Similarly, we have
\begin{align*}
H^{[qz,\epsilon,z]}(qz)=-\epsilon \dfrac{m_2v_2z}{m_1'v_1}.
\end{align*}
Since $ord_p(-\epsilon \dfrac{m_2v_2}{m_1'v_1})=0$ holds,
by  the previous argument $H^{[qz,\epsilon,z]}(qz)$ has a periodic $\Phi_1^{[\epsilon,z]}$ continued fraction expansion.

Finally, we suppose that
$ord_p(q)=-g<0$.
Let $m_2=m_2'p^g$, where $m_2'\in {\Bbb Z}$.
We suppose $k-g\geq 0$.
Then, we have
\begin{align*}
G^{[qz,\epsilon]}(qz)=\epsilon\dfrac{m_2p^{k-g}}{m_1z}-j=-\epsilon \dfrac{m_2'v_2z}{m_1v_1}-\epsilon\dfrac{m_2'uv_2}{m_1v_1}-j,
\end{align*}
where $j=\omega_p(\epsilon\frac{m_2p^{k-g}}{m_1z})$.
Similarly, we have
\begin{align*}
H^{[qz,\epsilon,z]}(qz)=-\epsilon \dfrac{m_2'v_2z}{m_1v_1}.
\end{align*}
Since $ord_p(-\epsilon \dfrac{m_2'v_2}{m_1v_1})=0$ holds,
$H^{[qz,\epsilon,z]}(qz)$ has a periodic $\Phi_1^{[\epsilon,z]}$ continued fraction expansion.
In the case of $k-g<0$, we have
\begin{align*}
H^{[qz,\epsilon,z]}(qz)=-\epsilon \dfrac{m_2'v_2p^{g-k}z}{m_1v_1}.
\end{align*}
Since $ord_p(-\epsilon \dfrac{m_2'v_2p^{g-k}}{m_1v_1})>0$ holds,
$H^{[qz,\epsilon,z]}(qz)$ has a periodic $\Phi_1^{[\epsilon,z]}$ continued fraction expansion.

\end{proof}

\begin{thm}\label{t2}
Let $K$ be a quadratic field over $\Bbb{Q}$.
Let $z\in K$ satisfy  Condition {\bf H} and $K={\Bbb Q}(z)$.
Then, every rational number $\alpha$ has a finite  $\Phi_1^{[\epsilon,z]}$ continued fraction expansion.
Every $\alpha\in K$ with $\alpha\notin \Bbb{Q}$
 has a periodic $\Phi_1^{[\epsilon,z]}$ continued fraction expansion.
\end{thm}

\begin{proof}
Let $x^2+ux+vp^k$ be the minimal polynomial of $z$, where $u,v\in {\Bbb Q}$, $ord_p(u)=ord_p(v)=0$ and $k>0$.
Let $v=\dfrac{v_1}{v_2}$, where $v_1, v_2$ are relatively prime integers.
Let $\alpha\ne 0$ be a rational number.
By Lemma \ref{l2.5} we may assume $\alpha\in p{\Bbb Z}_p$.
Let $\alpha=\frac{m_1p^{k_1}}{m_2}\ne 0$ such that $m_1,m_2$ are relatively prime integers with $ord_p(m_1)=ord_p(m_2)=0$ and
 $k_1\in {\Bbb Z}_{\geq 1}$.
Then, we see that $G^{[\overline{\alpha},\epsilon]}(\alpha)=\frac{-\epsilon m_2-jm_1}{m_1}$, where
$j=\omega_p(-\frac{\epsilon m_2}{m_1})$.
Since $\frac{-\epsilon m_2-jm_1}{m_1}\in p{\Bbb Z}_p\cap {\Bbb Q}$, $H^{[\overline{\alpha},\epsilon,z]}(\alpha)=0$.
Therefore, $\alpha$ has a finite  $\Phi_1^{[\epsilon,z]}$ continued fraction expansion.

Let $\alpha\in K$ be not rational.
Then,
 we have
\begin{align*}
H^{[\alpha,\epsilon,z]}(\alpha)=qz\ \text{or }H^{[\alpha,\epsilon,z]}(\alpha)=\dfrac{z}{m}+q',
\end{align*}
where $q,m,q'\in {\Bbb Q}$ with $|m|_p<1$ and $|q'|_p\geq 1$.
If $H^{[\overline{\alpha},\epsilon,z]}(\alpha)=qz$ holds, by Lemma \ref{l7},
we see that
$\alpha$ has a periodic  $\Phi_1^{[\epsilon,z]}$ continued fraction expansion.
We suppose that  $H^{[\alpha,\epsilon,z]}(\alpha)=\dfrac{z}{m}+q'$.
Let $k_2:=ord_p( \dfrac{z}{m}+q')>0$.
First, we suppose $ord_p(q')=0$.
Then, since  $ord_p(\dfrac{z}{m})=0$, we have  $ord_p(m)=k$.
Let $m=p^km'$, where $m'\in {\Bbb Q}$ and $ord_p(m')=0$.
Then, we get
\begin{align*}
 \dfrac{\epsilon}{\dfrac{z}{m}+q'}=
\dfrac{\epsilon m'p^k}{(m'p^kq')^2-um'p^kq'+p^kv}
(-z+m'p^kq'-u).
\end{align*}
We consider the coefficient of the term $z$ in the right hand side
of the above formula.
We have
\begin{align*}
&|(m'p^kq')^2-um'p^kq'+p^kv|_p=|(m'p^kq')^2-u(m'p^kq'+z)+uz+p^kv|_p\nonumber\\
&=|(m'p^kq')^2-u(m'p^kq'+z)-z^2|_p\nonumber\\
&=\left|(m'p^kq')^2-z^2-m'p^ku\left(q'+\dfrac{z}{m'p^k}\right)\right|_p\\
&=\left|(m'p^k)^2\left(q'+\dfrac{z}{m'p^k}\right)\left(q'-\dfrac{z}{m'p^k}\right)
-m'p^ku\left(q'+\dfrac{z}{m'p^k}\right)\right|_p\\
&=\left|m'p^ku\left(q'+\dfrac{z}{m'p^k}\right)\right|_p=p^{-(k+k_2)}.
\end{align*}
Therefore, we have
\begin{align}\label{eq1}
\left|\dfrac{-\epsilon m'p^{k+k_2}z}{(m'p^kq')^2-um'p^kq'+p^kv}\right|_p<1.
\end{align}
By (\ref{eq1}) and  $\displaystyle\left|\frac{\epsilon p^{k_2}}{\frac{z}{m}+q'}-\omega_p\left( \frac{\epsilon p^{k_2}}{\frac{z}{m}+q'}\right)\right|_p<1$,
we have
\begin{align*}
\left|\dfrac{\epsilon m'p^{k+k_2}}{(m'p^kq')^2-um'p^kq'+p^kv}
(m'p^kj-u)-\omega_p\left( \dfrac{\epsilon p^{k_2}}{\dfrac{z}{m}+q'}\right)\right|_p<1,
\end{align*}
which yields
\begin{align*}
\displaystyle H^{[\frac{z}{m}+q',\epsilon,z]}\left(\frac{z}{m}+q'\right)=
\dfrac{-\epsilon m'p^{k+k_2}}{(m'p^kq')^2-um'p^kq'+p^kv}z.
\end{align*}
By Lemma \ref{l7}, $\alpha$ has a periodic $\Phi_1^{[\epsilon,z]}$ continued fraction expansion.
Next, we suppose $ord_p(q')=-l<0$.
Then, since  $ord_p(\frac{z}{m})=-l$, we have  $ord_p(m)=k+l$.
Hence we can set  $m=p^{k+l}m'$ with $m'\in {\Bbb Q}$ and $ord_p(m')=0$.
Then, by the previous argument
there exists a rational number $q''\ne 0$ such that
\begin{align*}
\displaystyle H^{[\frac{z}{m'p^k}+q'p^l,\epsilon,z]}\left(\frac{z}{m'p^k}+q'p^l\right)
=q''z.
\end{align*}
Since $H^{[\frac{z}{m}+q',\epsilon,z]}(\frac{z}{m}+q')=H^{[p^l(\frac{z}{m}+q'),\epsilon,z]}(p^l(\frac{z}{m}+q'))$
holds, by Lemma \ref{l7},
$\alpha$ has a periodic $\Phi_1^{[\epsilon,z]}$ continued fraction expansion.
\end{proof}

By the proof of Theorem \ref{t2}
we have

\begin{cor}\label{c1}
Let $K$ be a quadratic field over $\Bbb{Q}$.
Let $z\in K$ satisfy Condition {\bf H} and $K={\Bbb Q}(z)$.
Let $x^2+ux+vp^k$ be a minimal polynomial of $z$, where $u,v\in {\Bbb Q}$
, $ord_p(u)=ord_p(v)=0$ and $k>0$.
Let $v=\dfrac{v_1}{v_2}$, where $v_1, v_2$ are relatively prime integers with $v_2>0$.
For $\alpha\in K$ with $\alpha\notin \Bbb{Q}$,
$\alpha$ has a purely periodic $\Phi_1^{[\epsilon,z]}$ continued fraction expansion if and only if
\begin{align*}
\alpha\in \{ qz|
\text{$q\in {\Bbb Q}$ and $ord_p(q)=0$ }\}.
\end{align*}
\end{cor}

\noindent
{Remark}.
We remark that if $K$ is a quadratic field, then
 $\Phi^{[\epsilon,z]}_1=\Phi^{[\epsilon,z],(n)}_2$
holds for $n\in {\Bbb Z}_{\geq 1}$.
Therefore, these continued fraction algorithms coincide with each other.

In the similar manner, we have

\begin{thm}\label{t3}
Let $K\subset {\Bbb Q}_p$ be a quadratic field over $\Bbb{Q}$.
Then, every rational number $\alpha$
has a finite  $\Phi_3^{[z]}$ continued fraction expansion.
For every $\alpha\in K$ with $\alpha\notin \Bbb{Q}$,
$\alpha$ has a periodic $\Phi_3^{[z]}$ continued fraction expansion.
\end{thm}

\section{Multidimensional Cases}
We can expect that higher dimensional $p$-adic versions of Lagrange's Theorem holds for some of our algorithms
from numerical experiments (see Section \ref{NE}), although we are not successful to give any proof at the moment.
 Dubois and Paysant-Le Roux \cite{D} showed that for every real
cubic number field there is a pair of numbers which has a periodic Jacobi-Perron expansion.
In this section, we show that  a $p$-adic version holds for the $\Phi_1^{[\epsilon,z]}$ continued fraction
 algorithm and there exist infinitely many elements depending on many parameters
 which have periodic $\Phi_3^{[\epsilon,z]}$ continued fraction expansions
for any finite extension of ${\Bbb Q}$  in ${\Bbb Q}_p$.

Let $K$ be a cubic field over ${\Bbb Q}$ and $K\subset {\Bbb Q}_p$.
Let $z\in K$ satisfy  Condition {\bf H} and $K={\Bbb Q}(z)$.
Since $mz$ satisfies Condition {\bf H} for an arbitrary integer $m$ which is relatively prime to $p$,
 we can choose $z$ which  is integral over  ${\Bbb Z}$.
Let $z$ be integral over  ${\Bbb Z}$ and
\begin{align}\label{eq2}
x^3+a_1x^2+a_2x+a_3p^k
\end{align}
be the minimal polynomial of $z$, where $a_i\in {\Bbb Z}$ for $1\leq i\leq 3$
, $ord_p(a_2)=ord_p(a_3)=0$ and $k\in {\Bbb Z}_{>0}$.

\begin{thm}\label{t4}
Let
$\overline{\alpha}:=(z,z^2)$.
Then,
$\overline{\alpha}$ has a periodic $\Phi_1^{[\epsilon,z]}$ continued fraction expansion.
\end{thm}

\begin{proof}
We have
\begin{align*}
&g_1^{[\overline{\alpha},\epsilon];(1)}(\overline{\alpha})=\dfrac{\epsilon p^k}{z}-\omega_p\left(\dfrac{\epsilon p^k}{z}\right)\\
&=\dfrac{-\epsilon z^2}{a_3}+\dfrac{-\epsilon a_1z}{a_3}+\dfrac{-\epsilon a_2}{a_3}-\omega_p\left(\dfrac{\epsilon p^k}{z}\right).
\end{align*}
We have
\begin{align*}
&h_1^{[\overline{\alpha},\epsilon,z];(1)}(\overline{\alpha})=\dfrac{-\epsilon z^2}{a_3}+\dfrac{-\epsilon a_1z}{a_3}+
\left\lfloor
\dfrac{-\epsilon a_2}{a_3}
-\omega_p\left(\dfrac{\epsilon p^k}{z}\right)
\right\rfloor_p.
\end{align*}
Since we see that$\displaystyle|h_1^{[\overline{\alpha},\epsilon,z];(1)}(\overline{\alpha})|_p<1$ and
$\displaystyle\left|\frac{-\epsilon z^2}{a_3}+\frac{-\epsilon a_1z}{a_3}\right|_p<1$,
we have
\begin{align*}
\left\lfloor
\dfrac{-\epsilon a_2}{a_3}
-\omega_p\left(\dfrac{\epsilon p^k}{z}\right)
\right\rfloor_p=0.
\end{align*}
Therefore, we have
\begin{align*}
h_1^{[\overline{\alpha},\epsilon,z];(1)}(\overline{\alpha})=\dfrac{-\epsilon z^2}{a_3}+\dfrac{-\epsilon a_1z}{a_3}.
\end{align*}
Thus, we have
\begin{align*}
\overline{\alpha}_1=
SH_1^{[\overline{\alpha},\epsilon,z]}(\overline{\alpha})
=\left(\epsilon z,\dfrac{-\epsilon z^2}{a_3}+\dfrac{-\epsilon a_1z}{a_3}\right).
\end{align*}
Similarly, we have
\begin{align*}
&\overline{\alpha}_2=
SH_1^{[\overline{\alpha}_1,\epsilon,z]}(\overline{\alpha}_1)
=\left(\dfrac{-\epsilon z}{a_3},\dfrac{- z^2}{a_3}+\dfrac{-a_1z}{a_3}\right),\\
&\overline{\alpha}_3=
SH_1^{[\overline{\alpha}_2,\epsilon,z]}(\overline{\alpha}_2)
=\left(z,z^2+a_1z\right),\\
&\overline{\alpha}_4=
SH_1^{[\overline{\alpha}_3,\epsilon,z]}(\overline{\alpha}_3)
=\left(\epsilon z,\dfrac{- \epsilon z^2}{a_3}+\dfrac{-\epsilon a_1z}{a_3}\right).
\end{align*}
Since $\overline{\alpha}_1=\overline{\alpha}_4$,
$\overline{\alpha}$ has a periodic $\Phi_1^{[\epsilon,z]}$ continued fraction expansion.

\end{proof}
Let $K\subset \Bbb{Q}_p$ be a finite extension of ${\Bbb Q}$
of arbitrary degree $>1$. Let $z\in K$ satisfy  Condition {\bf H} and $K={\Bbb Q}(z)$.
Let
\begin{align}\label{eq2}
x^{s+1}+a_1x^s+\ldots +a_{s}x+a_{s+1}p^k
\end{align}
be the minimal polynomial of $z$, where $a_i\in {\Bbb Q}\cap {\Bbb Z}_p$ for $1\leq i\leq s+1$
and $ord_p(a_s)=ord_p(a_{s+1})=0$.
We set $a_0:=1$.

\begin{thm}\label{t5}
Let  $u_i:=\sum_{i\leq j \leq s} a_{ij}z^{s-j+1}$ for $1\leq i\leq s-1$ and $u_s:=z$,
where $a_{ij}\in {\Bbb Q}\cap {\Bbb Z}_p$ for $1\leq i\leq s-1, i\leq j\leq s$
and $ord_p(a_{ii})\in U_p$ for $1\leq i\leq s-1$.
Then, $\overline{\alpha}:=(u_1,\ldots,u_s)$
 has a periodic $\Phi_3^{[z]}$ continued fraction expansion.
\end{thm}

\begin{proof}
We have
\begin{align*}
&g_s^{[\overline{\alpha},1];(s)}(\overline{\alpha})=\dfrac{p^k}{z}-\omega_p\left(\dfrac{p^k}{z}\right)\\
&=\sum_{0\leq i \leq s} \dfrac{-a_{s-i}z^{i}}{a_{s+1}}
-\omega_p\left(\dfrac{p^k}{z}\right).
\end{align*}
For $1\leq i\leq s-1$ we have
\begin{align*}
&g_s^{[\overline{\alpha},1];(i)}(\overline{\alpha})=\dfrac{u_i}{z}-\omega_p\left(\dfrac{u_i}{z}\right)\\
&=\sum_{i\leq j \leq s} a_{ij}z^{s-j}-\omega_p\left(\dfrac{u_i}{z}\right).
\end{align*}
Then, $M'_{\overline{\alpha}}\in M_{s}(\Bbb{Q})$ is given by
\begin{align*}
M'_{\overline{\alpha}}=
\begin{pmatrix}
0&a_{11}&\ldots&\ldots&a_{1s-1}\\
0&0&a_{22}&\ldots&a_{2s-1}\\
\vdots&\vdots&&\ddots&\vdots\\
0&0&\ldots&\ldots&a_{s-1s-1}\\
\dfrac{-a_{0}}{a_{s+1}}&\ldots&\ldots&\ldots&\dfrac{-a_{s-1}}{a_{s+1}}\\
\end{pmatrix}.
\end{align*}
We see that
every element of $M'_{\overline{\alpha}}$ is in ${\Bbb Z}_p$,
$ord_p(a_{ii})=0$ for $1\leq i\leq s-1$
and $ord_p(\dfrac{-a_{0}}{a_{s+1}})=0$.
Therefore, we see 
that $M'_{\overline{\alpha}}$ is converted to the unit matrix $I$ by the $p$-reduction algorithm.
Since $pr(M'_{\overline{\alpha}})M'_{\overline{\alpha}}=I$ holds, we have
\begin{align*}
\overline{\alpha}_2=pr(M'_{\overline{\alpha}})(h_1^{[\overline{\alpha},\epsilon,z];(1)}(\overline{\alpha}),\ldots,h_1^{[\overline{\alpha},\epsilon,z];(s)}(\overline{\alpha}))^{T}
=(z^s,z^{s-1},\ldots,z)^{T}.
\end{align*}
Similarly, we have $\overline{\alpha}_3=(z^s,z^{s-1},\ldots,z)^{T}$.
Therefore,
$\overline{\alpha}$
 has a periodic $\Phi_3^{[\epsilon,z]}$ continued fraction expansion.
\end{proof}

\section{Numerical Experiments}\label{NE}
In this section, we give some numerical results on our algorithms.
For the calculation of the tables, we used  computers equipped with GiNaC \cite{BFK} on GNU C++.

Concerning our experiments, $Height_z(\overline{\alpha})$, defined
below, plays an important role.
We suppose that $K={\Bbb Q}(z)$ with $z\in {\Bbb Q}_p$
satisfying Condition {\bf H}.
For $\alpha\in K$ with
\begin{align*}
\alpha=a_0+a_1z+\ldots +a_sz^{s} \ (a_0,\ldots,a_s\in {\Bbb Q}),
\end{align*}
we define
\begin{align*}
&Height_z(\alpha):=\max_{0\leq i\leq s} (Height(a_i)),\\
\end{align*}
where
$Height(a):=\max(|b|,|c|)$ for $a=b/c\neq 0 (b\in {\Bbb Z},\ c\in {\Bbb Z}_{>0})$ and
$Height(0):=0$.
For $\overline{\alpha}=(\alpha_1,\ldots,\alpha_{s})\in D=K^s$, we define
\begin{align*}
&Height_z(\overline{\alpha}):=\max_{1\leq i\leq s} (Height(\alpha_i)).\\
\end{align*}

For a given $\overline{\alpha}\in D=K^s$, we compute the
sequence $\{\overline{\alpha}_n\}$ according to a continued fraction algorithm $Algor$
until we find $n$ such that
\begin{align*}
&(\alpha \text{ in Case }\mathcal{P})\
\overline{\alpha}_n=\overline{\alpha}_m\ (0\leq \exists\  m<n)\ \&\
Height_z(\overline{\alpha}_n)\leq 10^{300},\\
&(\alpha  \text{ in Case } \mathcal{H})\
Height_z(\overline{\alpha}_n)> 10^{300}.
\end{align*}
We remark that the class $\mathcal{P}$ and $\mathcal{H}$
depend on $Algor$, which will be written
$\mathcal{P}(Algor)$, $\mathcal{H}(Algor)$
for some specified  algorithm $Algor$ later.\\

{\bf Example}.
Let $p=2$ and  $z\in p{\Bbb Z}_p$ be the root of $x^3+x-20=0$.

Let
$\overline{\alpha}:=(z,z^2)$.
Then, from the proof of Theorem \ref{t4}
$\overline{\alpha}$ has a periodic $\Phi_1^{[1,z]}$ continued fraction expansion with the length of the period $=3$ and
the length of the pre-period $=1$. We also see that
$Height_z(\overline{\alpha}_n)\leq 10^{300}$ for $0\leq n\leq 4$.
Hence, $\overline{\alpha}$ is in Case $\mathcal{P}$ according to the $\Phi_1^{[1,z]}$ continued fraction algorithm.
Let $\overline{\beta}:=(139/38-186z/113-122z^2/125,-193/158+17z/26-49z^2/93)$.
Then, we see that $Height_z(\overline{\beta}_6)> 10^{300}$ and
$\overline{\beta}$ is in Case $\mathcal{H}$ according to the $\Phi_1^{[1,z]}$ continued fraction algorithm.
\\

In Table 1, we give the  periodicity test of the
expansions obtained by the $\Phi_0^{[\epsilon,z]}$ continued fraction algorithm
for  $\alpha \in D = {\Bbb Q}(z) \subset {\Bbb Q}_p$,
where $\alpha$ runs over a subset of $D$ of 100 elements chosen by a pseudorandom
algorithm, $p$ runs over all the primes $<100$, and $z$ runs over the set
$\{z\in p{\Bbb Z}_p| x^2+ax+bp\text{\ is the minimal polynomial of }z, a,b\in {\Bbb Z},\ 0<a\leqq 10,-10\leq b\leq 10, ord_p(a)=0\}$.
We generate a set of 100 elements denoted by $Test_z$ in $D$  by using the pseudorandom algorithm given by Saito and Yamaguchi \cite{SY} as follows:

Let $0.d_1d_2\cdots$ be the binary expansion of the real positive root of $x^2+x-1$,
where $\{d_1,d_2\cdots\}$ is generated by the algorithm  \cite{SY} in which
they showed that the sequence has good properties as pseudorandom numbers.
We set
\begin{align*}
 &e_i=\sum_{k=1}^8 2^{k-1}d_{8i+k} \ \ \text{for $i\in {\Bbb Z}_{\geq 0}$},\\ &t_i=(-1)^{e_{6i+2}}\frac{e_{6i+1}}{e_{6i}+1}+(-1)^{e_{6i+5}}\frac{e_{6i+4}}{e_{6i+3}+1}z
 \ \ \text{for $i\in {\Bbb Z}_{\geq 0}$}.
\end{align*}
Let $m$ be the least integer such that
\begin{align*}
\sharp(\{t_i|i\leq m\}\backslash\{t_i|t_i\in {\Bbb Q}, i\leq m\})=100.
\end{align*}
We define $Test_z\subset D$ as
\begin{align*}
Test_z:=\{t_i|i\leq m\}\backslash\{t_i|t_i\in {\Bbb Q}, i\leq m\}.
\end{align*}

\newpage

\begin{center}
Table 1 $\Phi_0^{[\epsilon,z]}$ continued fraction algorithm with quadratic $z$
\end{center}

\begin{tabular}{|c|r|r|r|r|}
\hline
prime  &$1^*$&$2^*$&$3^*$ & $4^*$\\
number&&&&\\
\hline
2&0&7800&0&7800\\
\hline
3&0&11700&0&11700\\
\hline
5&0&14400&1&14399\\
\hline
7&0&16600&2&16598\\
\hline
11&0&19000&0&19000\\
\hline
13&0&19200&0&19200\\
\hline
17&0&9600&0&19600\\
\hline
19&0&19800&0&19800\\
\hline
23&0&20000&0&20000\\
\hline
29&0&20000&0&20000\\
\hline
31&0&20000&0&20000\\
\hline
37&0&20000&0&20000\\
\hline
41&0&20000&0&20000\\
\hline
\end{tabular}
\begin{tabular}{|c|r|r|r|r|}
\hline
prime&\ $1^*$& $2^*$&\ $3^*$& $4^*$\\
number&&&&\\
\hline
43&0&20000&0&20000\\
\hline
47&0&20000&0&20000\\
\hline
53&0&20000&0&20000\\
\hline
59&0&20000&0&20000\\
\hline
61&0&20000&0&20000\\
\hline
67&0&20000&0&20000\\
\hline
71&0&20000&0&20000\\
\hline
73&0&20000&0&20000\\
\hline
79&0&20000&0&20000\\
\hline
83&0&20000&0&20000\\
\hline
89&0&20000&0&20000\\
\hline
97&0&20000&0&20000\\
\hline
&&&&\\
\hline
\end{tabular}
\\$1^*$ (resp., $3^*$)is the number of elements $\alpha$ in $\mathcal{P}(\Phi_0^{[1,z]})$
(resp., $\mathcal{P}(\Phi_0^{[-1,z]})$) and
$2^*$(resp., $4^*$) is
the number of elements $\alpha$ in $\mathcal{H}(\Phi_0^{[1,z]})$
(resp., $\mathcal{H}(\Phi_0^{[-1,z]})$).
\\

In Table 2 for the prime numbers $p$ with $2\leq p \leq 100$ and $\{z\in p{\Bbb Z}_p| x^3+ax+bp\text{\ is the minimal polynomial of }z, a,b\in {\Bbb Z},\ 0<a\leqq 10,-10\leq b\leq 10, ord_p(a)=0\}$
and the set of 100 elements in $D$ denoted by $Test_z^{(2)}$ (described later), we observe periodicity
by the $\Phi_1^{[\epsilon,z]}$  continued fraction algorithm.
Let $0.d_1^{(1)}d_2^{(1)}\cdots$
(resp., $0.d_1^{(2)}d_2^{(2)}\cdots$) be the binary expansion of the real positive root of $x^2+2x-1$(resp., $x^2+2x-2$),
where $\{d_1^{(j)},d_2^{(j)}\cdots\}$ for $j=1,2$ are generated by the algorithm  \cite{SY}.
We set that for $j=1,2$
\begin{align*}
 &e_i^{(j)}=\sum_{k=1}^8 2^{k-1}d_{8i+k}^{(j)} \ \ \text{for $i\in {\Bbb Z}_{\geq 0}$},\\
&t_i^{(j)}=(-1)^{e_{9i+2}^{(j)}}\frac{e_{9i+1}^{(j)}}{e_{9i}^{(j)}+1}+(-1)^{e_{9i+5}^{(j)}}
\frac{e_{9i+4}^{(j)}}{e_{9i+3}^{(j)}+1}z
+(-1)^{e_{9i+8}^{(j)}}
\frac{e_{9i+7}^{(j)}}{e_{9i+6}^{(j)}+1}z^2
 \ \ \text{for $i\in {\Bbb Z}_{\geq 0}$}.
\end{align*}
Let $m$ be the least integer such that
\begin{align*}
&\sharp(\{(t_i^{(1)},t_i^{(2)})|i\leq m\}\backslash\{(t_i^{(1)},t_i^{(2)})|
\text{$1,t_i^{(1)}, t_i^{(2)}$ are linearly dependent over ${\Bbb Q}$}, i\leq m\})\\
&=100.
\end{align*}
We define $Test_z^{(2)}\subset D$ as
\begin{align*}
Test_z^{(2)}:=
\{(t_i^{(1)},t_i^{(2)})|i\leq m\}\backslash\{(t_i^{(1)},t_i^{(2)})|
\text{$1,t_i^{(1)}, t_i^{(2)}$ are linearly dependent over ${\Bbb Q}$}, i\leq m\}.
\end{align*}

\newpage

\begin{center}
Table 2　$\Phi_1^{[-1,z]}$ continued fraction algorithm with cubic $z$
\end{center}
\begin{tabular}{|c|r|r|r|r|}
\hline
prime&\ $1^*$& $2^*$&\ $3^*$& $4^*$\\
number&&&&\\
\hline
2&6767&1633&6773&1627\\
\hline
3&10328 &2272 &10340 &2260 \\
\hline
5&11656 &3144 &11660 &3140 \\
\hline
7&13982 &3218 &13991 &3209 \\
\hline
11&15451 &3749 &15458 &3742\\
\hline
13&16553 &2647 &16535 &2665 \\
\hline
17&17424 &1976 &17405 &1995\\
\hline
19&17358 &2242 &17367 &2233\\
\hline
23&17596 &2204 &17600 &2200\\
\hline
29&16911 &2889 &16915 &2885 \\
\hline
31&18435 &1365 &18432 &1368 \\
\hline
37&18795 &1005 &18795 &1005 \\
\hline
41&17750 &2050 &17742 &2058\\
\hline
\end{tabular}
\begin{tabular}{|c|r|r|r|r|}
\hline
prime&\ $1^*$& $2^*$&\ $3^*$& $4^*$\\
number&&&&\\
\hline
43&18897 &903 &18904 &896\\
\hline
47&18889 &1111 &18882 &1118 \\
\hline
53&18178 &1622 &18188 &1612 \\
\hline
59&18805 &995 &18807 &993 \\
\hline
61&19396 &604 &19396 &604 \\
\hline
67&19317 &483 &19314 &486 \\
\hline
71&19548 &252 &19546 &254 \\
\hline
73&18710 &1090 &18716 &1084 \\
\hline
79&19153 &847 &19156 &844 \\
\hline
83&19622 &178 &19620 &180 \\
\hline
89&19340 &460 &19339 &461 \\
\hline
97&19052 &948 &19046 &954\\
\hline
&&&&\\
\hline
\end{tabular}
\\$1^*$ (resp., $3^*$)is the number of elements $\alpha$ in $\mathcal{P}(\Phi_1^{[1,z]})$
(resp., $\mathcal{P}(\Phi_1^{[-1,z]})$) and
$2^*$(resp., $4^*$) is
the number of elements $\alpha$ in $\mathcal{H}(\Phi_1^{[1,z]})$
(resp., $\mathcal{H}(\Phi_1^{[-1,z]})$).
\\

In Table 3 for the prime numbers $p$ with $2\leq p \leq 100$ and $\{z\in p{\Bbb Z}_p| x^3+ax+bp\text{\ is the minimal polynomial of }z, a,b\in {\Bbb Z},\ 0<a\leqq 10,-10\leq b\leq 10, ord_p(a)=0\}$
and 100 elements in $D$ given in the same way as Table 2 we observe periodicity
by the  $\Phi_2^{[\epsilon,z],(1)}$ continued fraction algorithm.

\begin{center}
Table 3 $\Phi_2^{[-1,z],(1)}$ continued fraction algorithm with cubic $z$
\end{center}

\begin{tabular}{|c|r|r|r|r|}
\hline
prime&\ $1^*$& $2^*$&\ $3^*$& $4^*$\\
number&&&&\\
\hline
2&7568 &832 &7586 &814\\
\hline
3&11925 &675 &11924 &676 \\
\hline
5&13396 &1404 &13403 &1397\\
\hline
7&15865 &1335 &15863  &1337\\
\hline
11& 17968&1232  &17968 &1232 \\
\hline
13&18045 &1155  &18028  &1172\\
\hline
17&18209 &1191 &18209 &1191\\
\hline
19&18784 &816 &18785 &815 \\
\hline
23&19279 &521 &19288 &512\\
\hline
29&19020 &780 &19019 &781 \\
\hline
31&19500 &300 &19497 &303\\
\hline
37&19158 &642 &19164 &636\\
\hline
41&19074 &726 &19072 &728 \\
\hline
\end{tabular}
\begin{tabular}{|c|r|r|r|r|}
\hline
prime&\ $1^*$& $2^*$&\ $3^*$& $4^*$\\
number&&&&\\
\hline
43&19168 &632 &19180 &620 \\
\hline
47&19768 &232 &19763 &237 \\
\hline
53&19297 &503 &19297 &503 \\
\hline
59&19560 &240 &19556 &244 \\
\hline
61&19658 &342 &19657 &343\\
\hline
67&19502 &298 &19503 &297 \\
\hline
71&19751 &49 &19755 &45 \\
\hline
73&19476 &324 &19477 &323 \\
\hline
79&19914 &86 &19915 &85 \\
\hline
83&19635 &165 &19635 &165 \\
\hline
89&19510 &290 &19510 &290 \\
\hline
97&19623 &377 &19626 &374 \\
\hline
&&&&\\
\hline
\end{tabular}
\\$1^*$ (resp., $3^*$)is the number of elements $\alpha$ in $\mathcal{P}(\Phi_2^{[1,z],(1)})$
(resp., $\mathcal{P}(\Phi_2^{[-1,z],(1)}$) and
$2^*$(resp., $4^*$) is
the number of elements $\alpha$ in $\mathcal{H}(\Phi_2^{[1,z],(1)})$
(resp., $\mathcal{H}(\Phi_2^{[-1,z],(1)})$).
\\

In Table 4 for the prime numbers $p$ with $2\leq p \leq 100$ and $\{z\in p{\Bbb Z}_p| x^3+ax+bp\text{\ is the minimal polynomial of }z, a,b\in {\Bbb Z},\ 0<a\leqq 10,-10\leq b\leq 10, ord_p(a)=0\}$
and 100 elements in $D$ given in the same way as Table 2 we observe periodicity
by the  $\Phi_2^{[\epsilon,z],(2)}$ continued fraction algorithm.

\begin{center}
Table 4 $\Phi_2^{[-1,z],(2)}$ continued fraction algorithm with cubic $z$
\end{center}

\begin{tabular}{|c|r|r|r|r|}
\hline
prime  &$1^*$&$2^*$&$3^*$ & $4^*$\\
number&&&&\\
\hline
2&8221&179&8215&185\\
\hline
3&12494&106&12490&110\\
\hline
5&14633 &167 &14636 &164 \\
\hline
7&16988 &212 &16993 &207 \\
\hline
11&19102 &98 &19102 &98 \\
\hline
13&19063 &137 &19052 &148 \\
\hline
17&19273&127&19275 &125 \\
\hline
19&19516 &84 &19515 &85 \\
\hline
23&19718 &82 &19719 &81 \\
\hline
29&19717&83 &19712 &88 \\
\hline
31&19775 &25 &19775 &25 \\
\hline
37&19730 &70 &19727 &73 \\
\hline
41&19755 &45 &19758 &42 \\
\hline
\end{tabular}
\begin{tabular}{|c|r|r|r|r|}
\hline
prime  &$1^*$&$2^*$&$3^*$ & $4^*$\\
number&&&&\\
\hline
43&19766 &34 &19762 &38 \\
\hline
47&19984 &16 &19985 &15 \\
\hline
53&19699 &101 &19700 &100 \\
\hline
59&19726 &74 &19724 &76 \\
\hline
61&19986 &14 &19983 &17 \\
\hline
67&19795 &5 &19795 &5 \\
\hline
71&19781 &19 &19781 &19 \\
\hline
73&19784 &16 &19783 &17 \\
\hline
79&20000 &0 &20000 &0 \\
\hline
83&19766 &34 &19765 &35 \\
\hline
89&19789 &11 &19786 &14 \\
\hline
97&19971 &29 &19968 &32 \\
\hline
&&&&\\
\hline
\end{tabular}
\\$1^*$ (resp., $3^*$)is the number of elements $\alpha$ in $\mathcal{P}(\Phi_2^{[1,z],(2)})$
(resp., $\mathcal{P}(\Phi_2^{[-1,z],(2)}$) and
$2^*$(resp., $4^*$) is
the number of elements $\alpha$ in $\mathcal{H}(\Phi_2^{[1,z],(2)})$
(resp., $\mathcal{H}(\Phi_2^{[-1,z],(2)})$).
\\

In Table 5 for the prime numbers $p$ with $2\leq p \leq 100$ and $\{z\in p{\Bbb Z}_p| x^3+ax+bp\text{\ is the minimal polynomial of }z, a,b\in {\Bbb Z},\ 0<a\leqq 10,-10\leq b\leq 10, ord_p(a)=0\}$
and 100 elements in $D$ given in the same way as Table 2 we observe periodicity
by the  $\Phi_2^{[\epsilon,z],(3)}$ continued fraction algorithm.

\newpage

\begin{center}
Table 5 $\Phi_2^{[-1,z],(3)}$ continued fraction algorithm with  cubic $z$
\end{center}

\begin{tabular}{|c|r|r|r|r|}
\hline
prime  &$1^*$&$2^*$&$3^*$ & $4^*$\\
number&&&&\\
\hline
2&8380&20&8380&20\\
\hline
3&12580&20&12581&19\\
\hline
5&14793&7&14795&5\\
\hline
7&17164&36&17164&36\\
\hline
11&19197&3&19197&3\\
\hline
13&19200&0&19200&0\\
\hline
17&19384&16&19384&16\\
\hline
19&19599&1&19599&1 \\
\hline
23&19800&0&19800&0 \\
\hline
29&19800&0&19800&0 \\
\hline
31&19800&0&19800&0 \\
\hline
37&19800&0&19800&0 \\
\hline
41&19800&0&19800&0 \\
\hline
\end{tabular}
\begin{tabular}{|c|r|r|r|r|}
\hline
prime  &$1^*$&$2^*$&$3^*$ & $4^*$\\
number&&&&\\
\hline
43&19800&0&19800&0 \\
\hline
47&20000&0&20000&0 \\
\hline
53&19800&0&19800&0 \\
\hline
59&19800&0&19800&0 \\
\hline
61&20000&0&20000&0 \\
\hline
67&19800&0&19800&0 \\
\hline
71&19800&0&19800&0 \\
\hline
73&19800&0&19800&0 \\
\hline
79&20000&0&20000&0 \\
\hline
83&19800&0&19800&0 \\
\hline
89&19800&0&19800&0 \\
\hline
97&20000&0&20000&0 \\
\hline
&&&&\\
\hline
\end{tabular}
\\$1^*$ (resp., $3^*$)is the number of elements $\alpha$ in $\mathcal{P}(\Phi_2^{[1,z],(2)})$
(resp., $\mathcal{P}(\Phi_2^{[-1,z],(3)}$) and
$2^*$(resp., $4^*$) is
the number of elements $\alpha$ in $\mathcal{H}(\Phi_2^{[1,z],(3)})$
(resp., $\mathcal{H}(\Phi_2^{[-1,z],(3)})$).
\\

In Table 6 for the prime numbers $p$ with $2\leq p \leq 100$, $deg\in \{3,4,5,6\}$ and $\{z\in p{\Bbb Z}_p| x^{deg}+ax+bp\text{\ is the minimal polynomial of }z, a,b\in {\Bbb Z},\ 0<a\leqq 10,-10\leq b\leq 10, ord_p(a)=0\}$
and 100 elements in $D$ given in the same way as Table 2 we observe periodicity
by the  $\Phi_{3}^{[z]}$ continued fraction algorithm.

\newpage

\begin{center}
Table 6 $\Phi_{3}^{[z]}$ continued fraction algorithm with $3\leq$ the degree of $z \leq 6$
\end{center}

\begin{center}
\begin{tabular}{|c|r|r|r|r|r|r|r|r|}
\hline
&\multicolumn{2}{|c|}{$deg=3$}&\multicolumn{2}{|c|}{$deg=4$}&\multicolumn{2}{|c|}{$deg=5$}&\multicolumn{2}{|c|}{$deg=6$}\\
\hline
prime  &$1^*$&$2^*$&$1^*$&$2^*$&$1^*$&$2^*$&$1^*$&$2^*$\\
number&&&&&&&&\\
\hline
2&8400&0&8100&0&8800&0&9000&0\\
\hline
3&12600&0&12900&0&13200&0&13400&0\\
\hline
5&14800&0&15200&0&15200&0&15500&0\\
\hline
7&17200&0&17400&0&17600&0&17600&0\\
\hline
11&19200&0&19600&0&19600&0&19600&0\\
\hline
13&19200&0&19800&0&19800&0&19800&0\\
\hline
17&19400&0&19700&0&19800&0&19900&0\\
\hline
19&19600&0&19900&0&19800&0&19900&0\\
\hline
23&19800&0&19900&0&19800&0&19900&0\\
\hline
29&19800&0&19800&0&19800&0&19900&0\\
\hline
31&19800&0&19800&0&20000&0&19900&0\\
\hline
37&19800&0&19800&0&20000&0&19900&0\\
\hline
41&19800&0&20000&0&19800&0&19900&0\\
\hline
43&19800&0&20000&0&19800&0&20000&0\\
\hline
47&20000&0&20000&0&20000&0&20000&0\\
\hline
53&19800&0&20000&0&20000&0&20000&0\\
\hline
59&19800&0&19800&0&20000&0&20000&0\\
\hline
61&20000&0&19800&0&20000&0&20000&0\\
\hline
67&19800&0&19800&0&20000&0&20000&0\\
\hline
71&19800&0&19900&0&20000&0&20000&0\\
\hline
73&19800&0&19900&0&20000&0&20000&0\\
\hline
79&20000&0&20000&0&20000&0&19900&0\\
\hline
83&19800&0&20000&0&19800&0&19900&0\\
\hline
89&19800&0&20000&0&19800&0&20000&0\\
\hline
97&20000&0&20000&0&20000&0&20000&0\\
\hline
\end{tabular}
\end{center}
$1^*$ is the number of elements $\alpha$ in $\mathcal{P}(\Phi_{3}^{[z]})$
$2^*$ is
the number of elements $\alpha$ in $\mathcal{H}(\Phi_{3}^{[z]})$
\\

\phantom{}

\section{Conjecture}

We give the following conjectures which are supported by our numerical
experiments.\\

\noindent
{\bf Conjecture 1}.
Let $p$ be any prime number, and $K$ be any finite extension of ${\Bbb Q}$
with  $K\subset {\Bbb Q}_p$.
Let $s+1$ be its degree over ${\Bbb Q}$, $z\in K$ be any element satisfying Condition {\bf H} and
$K={\Bbb Q}(z)$.
For every  $\overline{\alpha}=(\alpha_1,\ldots, \alpha_s)\in K^s$
such that $1, \alpha_1,\ldots, \alpha_s$
are linearly independent over ${\Bbb Q}$,
$\overline{\alpha}$
has a periodic $\Phi_{3}^{[z]}$  continued fraction expansion.\\

\noindent
{\bf Conjecture 2}.
Let $p$ be any prime number, and $K$ be any cubic extension of ${\Bbb Q}$
with  $K\subset {\Bbb Q}_p$.
Let $z\in K$ be any element satisfying Condition {\bf H} and
$K={\Bbb Q}(z)$ and $\epsilon\in \{-1,1\}$.
There exists a map $\phi: D\to Ind$ such that
for every  $\overline{\alpha}=(\alpha_1, \alpha_2)\in K^2$
such that $1, \alpha_1, \alpha_2$
are linearly independent over ${\Bbb Q}$,
$\overline{\alpha}$
has a periodic $\Phi$  continued fraction expansion, where
$\Phi$ is the $c$-map $\Phi=(\phi(\cdot),H{}_{\phi(\cdot)}^{[\cdot,\epsilon,z]},id,\bar{0})$.\\

We remark that  Conjecture 1 holds for $s=1$ (Theorem \ref{t3}).

\begin{center}
{\bf Acknowledgements}
\end{center}

The authors would like to thank the anonymous referees for several helpful comments and remarks.
This research was supported by JSPS KAKENHI Grant Number 15K00342.

\end{document}